\documentclass{article}
\usepackage{tikz}
\usetikzlibrary{matrix}

\usepackage{amssymb}
\usepackage{verbatim}
\usepackage{array}
\usepackage{latexsym}
\usepackage{enumerate}
\usepackage{amsmath}
\usepackage{amsfonts}
\usepackage{amsthm}
\usepackage{color}
\usepackage[english]{babel}

\newtheorem{theorem}{Theorem}[section]

\newtheorem{lemma}[theorem]{Lemma}

\def\cC{\mathcal C}

\def\cF{\mathcal F}

\def\cH{\mathcal H}

\def\cX{\mathcal X}
\def\cY{\mathcal Y}

\def\PG{{\rm{PG}}}

\def\deg{\mbox{\rm deg}}

\def\dim{\mbox{\rm dim}}

\def\fq{{\mathbb F}_q}

\def\gg{\mathfrak{g}}

\newcommand{\PGL}{\mbox{\rm PGL}}

\newcommand{\aut}{\mbox{\rm Aut}}

\title{The Geometry of the Artin-Schreier-Mumford Curves over an Algebraically Closed Field}
\date{}
\author{G\'abor Korchm\'aros, Maria Montanucci}
\begin{document}
\maketitle

\vspace{0.5cm}\noindent {\em Keywords}:
Algebraic curves, algebraic function fields, positive characteristic, automorphism groups.
\vspace{0.2cm}\noindent

\vspace{0.5cm}\noindent {\em Subject classifications}:
\vspace{0.2cm}\noindent  14H37, 14H05.


\begin{abstract} For a power $q$ of a prime $p$, the Artin-Schreier-Mumford curve $ASM(q)$ of genus $\gg=(q-1)^2$ is the nonsingular model $\cX$ of the irreducible plane curve with affine equation $(X^q+X)(Y^q+Y)=c,\, c\neq 0,$ defined over a field $\mathbb{K}$ of characteristic $p$. The Artin-Schreier-Mumford curves are known from the study of algebraic curves defined over a non-Archimedean valuated field since for $|c|<1$ they are curves with a large solvable automorphism group of order $2(q-1)q^2 =2\sqrt{\gg}(\sqrt{\gg}+1)^2$, far away from the Hurwitz bound $84(\gg-1)$ valid in zero characteristic; see \cite{Co-Ka2003-1,Co-Ka-Ko2001,Co-Ka2004}. In this paper we deal with the case where $\mathbb{K}$ is an algebraically closed field of characteristic $p$. We prove that the group $\aut(\cX)$ of all automorphisms of $\cX$ fixing $\mathbb{K}$ elementwise has order $2q^2(q-1)$ and it is the semidirect product $Q\rtimes D_{q-1}$ where $Q$ is an elementary abelian group of order $q^2$ and $D_{q-1}$ is a dihedral group of order $2(q-1)$. For the special case $q=p$, this result was proven by Valentini and Madan \cite{mv1982}; see also \cite{AK}. Furthermore, we show that $ASM(q)$ has a nonsingular model $\cY$ in the three-dimensional projective space $PG(3,\mathbb{K})$ which is neither classical nor Frobenius classical over the finite field $\mathbb{F}_{q^2}$.
\end{abstract}
\section{Introduction}
\label{intr}
In this paper we work out a purely geometric approach to the study of Artin-Schreier-Mumford curves over an algebraically closed field $\mathbb{K}$ of characteristic $p$. We construct a nonsingular model $\cY$ of $ASM(q)$ of degree $2q$ in the three-dimensional projective space $\PG(3,\mathbb{K})$ and show that the linear series $g_{2q}^3$ cut out on $\cY$ by planes is complete. As a consequence, the automorphism group $\aut(ASM(q))$ of $ASM(q)$ fixing $\mathbb{K}$ elementwise is linear in $\PG(3,\mathbb{K})$ as it is isomorphic to the subgroup of $\PGL(4,\mathbb{K})$ which leaves the pointset of $\cY$ invariant. Using this isomorphism we prove that $\aut(ASM(q))$ is a solvable group of order $2(q-1)q^2=2\sqrt{\gg}(\sqrt{\gg}+1)^2$ as it happens over a non-Archimedean valuated field for $|c|<1$.
 In particular, $\aut(ASM(q))$ is quite large compared with its genus. In fact, if $\cX$ is any ordinary curve of genus $\gg(\cX)\geq 2$ defined over $\mathbb{K}$ of odd characteristic, and  $G$ is a solvable group of automorphisms of $\cX$ fixing $\mathbb{K}$ elementwise, then $|G|\leq 34(\gg(\cX)+1)^{3/2}$; see \cite{KM}.

We also point out that our nonsingular model $\cY$ has some unusual properties, namely it is neither classical nor Frobenius classical over $\mathbb{F}_{q^2}$.

\section{Background and Preliminary Results}\label{sec2}
The prime field of $\mathbb{K}$ is the finite field $\mathbb{F}_p$ and $\mathbb{K}$ contains a subfield $\mathbb{F}_{p^i}$ for every $i\geq 1$. In particular, $\mathbb{F}_q$ and its quadratic extension $\mathbb{F}_{q^2}$ are subfields of $\mathbb{K}$, and the trace function ${\rm{Tr}}(x)=x^q+x$ is defined for every $x\in \mathbb{K}$.

By a little abuse of notation, the term of Artin-Schreier-Mumford curve is also used for its plane model $\cC_q$ given by the curve of degree $2q$ with affine equation
\begin{equation}
\label{eq17dic2016} (X^q+X)(Y^q+Y)=c,\, c\in \mathbb{K}^*.
\end{equation}
We collect some basic facts on Artin-Schreier-Mumford curves. The curve $\cC_q$ has exactly two singular points, namely $X_\infty$ and $Y_\infty$, both have multiplicity $q$. They are ordinary singularities, and $\cC_q$ has genus $\gg=(q-1)^2$. The tangents at $X_\infty$ are the horizontal lines with equation $Y=d$ with ${\rm{Tr}}(d)=0$, and there exist exactly $q$ (linear) branches of $\cC_q$ centered at $X_\infty$. They form a set $\Omega_1$ of size $q$.
Similarly, the tangents at $Y_\infty$ are the vertical lines with equation $X=d$ with ${\rm{Tr}}(d)=0$, and the set $\Omega_2$ of the (linear) branches of $\cC_q$ centered at $Y_\infty$ has size $q$.

The following linear maps are automorphisms of $\cC_q$
\begin{equation}
\label{eq27dic2016}
\varphi_ {\alpha,\beta,v}(X,Y)=(vX+\alpha,v^{-1}Y+\beta), \quad \alpha,\beta,v \in \mathbb{K}, \quad {\rm{Tr(\alpha)}}={\rm{Tr(\beta)}}=0, \quad v^{q-1}=1.
\end{equation}
Here, $\Delta=\{\varphi_{\alpha,\beta,1}|\alpha,\beta\in \mathbb{K},{\rm{Tr(\alpha)}}={\rm{Tr(\beta)}}=0\}$ is an elementary abelian group of order $q^2$. Its subgroup $\Delta_1=\{\varphi_{\alpha,0,1}|\alpha\in \mathbb{K},{\rm{Tr(\alpha)}}=0\}$ has order $q$ and acts as a sharply transitive permutation group on the set of all tangents to $\cC_q$ at $Y_\infty$ (and hence on the set of the branches centered at $Y_\infty$) while it leaves invariant every tangent at $X_\infty$ (and hence every branch centered at $X_\infty$). Analogous results hold for the subgroup $\Delta_2=\{\varphi_{0,\beta,1}|\beta\in \mathbb{K},{\rm{Tr(\beta)}}=0\}$ whenever the roles of $X_\infty$ and $Y_\infty$ are interchanged. Also, $C=\{\varphi_{0,0,v}|v^{q-1}=1\}$ is a cyclic group of order $q-1$, and $\Delta$ together with $C$ generate a group $\Delta\rtimes C$ of order $q^2(q-1)$.

A further linear map which is an automorphism of $\cC_q$ is the involution
\begin{equation}
\label{eq37dic2016}
\xi(X,Y)=(Y,X)
\end{equation}
which interchanges $X_\infty$ and $Y_\infty$.  It fixes each point $P(a,a)$ of $\cC_q$ with ${\rm{Tr}}(a)^2=c$. The group generated by $C$ and $\xi$ is a dihedral group $D_{q-1}$ of order $2(q-1)$, and the group generated by the above linear maps is the semidirect product $\Delta\rtimes D_{q-1}$. Each linear map which is an automorphism of $\cC_q$ fixing $\mathbb{K}$ elementwise belongs to
$\Delta\rtimes D_{q-1}$.

Up to a field isomorphism over $\mathbb{K}$, the function field $\mathbb{K}(\cX)$ is $\mathbb{K}(x,y)$ with $(x^q+x)(y^q+y)-c=0$. In particular, $\Delta\rtimes D_{q-1}$ is a subgroup of $\aut(\mathbb{K}(\cX))$. The $q$ branches centered at $X_\infty$ correspond to $q$ places of $\mathbb{K}(\cX)$, say $P_1,\ldots,P_q$, and $\Delta\rtimes C$ leaves $\Omega_1=\{P_1,\ldots, P_q\}$ invariant. Analogously, $\mathbb{K}(\cX)$ has $q$ places, $Q_1,\ldots,Q_q$, arising from the branches centered  at $Y_\infty$, and $\Omega_2=\{Q_1,\ldots,Q_q\}$ is left invariant by $\Delta\rtimes C$. The involution $\xi$ interchanges $\Omega_1$ with $\Omega_2$.

If $\cC_q$ is defined over $\mathbb{F}_q$, that is, $c\in \mathbb{F}_q^*$, we may assume $c=1$. For this particular case the following results hold.

If $p$ is odd, a (nonsingular) point $P(u,v)\in \cC_q$ is defined over $\mathbb{F}_{q^2}$ if and only if  $u=\alpha+i\beta,v=\gamma+i\delta$ with $\alpha,\beta,\gamma,\delta\in \mathbb{F}_q$ where $i^2=s$ with a non-square element $s$ of $\mathbb{F}_q$ and  $i^q=-i$. From (\ref{eq17dic2016}), $(u^q+u)(v^q+v)=4\alpha\gamma=1$ whence $\gamma=\textstyle\frac{1}{4}(\alpha)^{-1}$. The number of $\mathbb{F}_{q^2}$-rational points of $\cC_q$ (and $\cX$) counts $(q-1)q^2+2q$ where the term $2q$ counts the number of linear branches centered at $X_\infty$ or $Y_\infty$.

If $p=2$,  a (nonsingular) point $P(u,v)\in \cC_q$ is defined over $\mathbb{F}_{q^2}$ if and only if  $u=\alpha+i\beta,v=\gamma+i\delta$ with $\alpha,\beta,\gamma,\delta\in \mathbb{F}_q$ where $i^2=i+s$ with a category one element $s\in \mathbb{F}_q$, and $i^q=i+s$; see \cite{Hirschfeld}. From (\ref{eq17dic2016}), $(u^q+u)(v^q+v)=\beta\delta=c$ whence $\delta=\beta^{-1} c$. As for $p>2$, the number of $\mathbb{F}_{q^2}$-rational points of $\cC_q$ (and $\cX$) equals $(q-1)q^2+2q$.

Take a point $P(a,a)\in \cC_q$ with $a\in\mathbb{F}_{q^2}$. Then the stabilizer of $P(a,a)$ in $\Delta\rtimes C$ is trivial. Therefore, $\Delta\rtimes C$ acts on the set $\Sigma$ of affine $\mathbb{F}_{q^2}$-rational points as a sharply transitive permutation group.
In terms of $\cX$, the set of all $\mathbb{F}_{q^2}$-rational points of $\cX(\mathbb{F}_{q^2})$ splits into two orbits under the action of $\Delta\rtimes D_{q-1}$, namely $\Omega_1\cup\Omega_2$ and $\Sigma$.

\section{Non-classicality of the Artin-Schreier-Mumford curve with respect to conics}
The theory of non-classical curves is found in \cite[Chapters 7,8]{HKT}. Here we point out that if $q\geq 5$ then the Artin-Schreier-Mumford curves are both non-classical and Frobenius non-classical with respect to conics.
\begin{lemma}
\label{lem58dic2016} Assume that $q\geq 5$. With respect to conics, the Artin-Schreier-Mumford curve $\cC_q$ is non-classical and has order sequence $(0,1,2,q).$
\end{lemma}
\begin{proof} Let $z_0=1+XY,\,z_1=Y,z_2=X,z_3=0,z_4=1,z_5=0$. The left hand side in Equation (\ref{eq17dic2016}) can be written as $z_0^q+z_1^qX+z_2^qY+z_3^qX^2+z_4^qXY+z_5Y^2$, and the first assertion follows from \cite[Section 7.8]{HKT}. Take an affine point $P=(u,v)\in \cC_q$. A primitive representation of the (unique) branch of $\cC_q$ centered at $P$ is
\begin{equation} \label{coeff} X=u+t,Y=v+v_1t+\ldots+v_it^i+\ldots,.
\end{equation}
where
\begin{equation}
\label{coefq2}
v_i=\left\{
\begin{array}{lll}
{\mbox{$(-1)^i\frac{v^q+v}{(u^q+u)^i}$  for  $1 \leq i\le q+1$, $i\neq q$}};\\
{\mbox{$(-1)^q\frac{v^q+v}{(u^q+u)^q}-(v_1^q+v_1)$  for  $i=q$.}}
\end{array}
\right.
\end{equation}
In particular,
\begin{equation} \label{coefq} v^q+v+v_q(u^q+u)+v_1^q(u^q+u)+v_{q-1}=0, \end{equation} and \begin{equation} \label{coefq1} v_1+v_1^q+v_q+v_{q+1}(u^q+u)=0. \end{equation}

From this, the local expansion of $z_0,z_1,z_2,z_4$ at $P$ are $z_0=1+uv+(v+v_1u)t+\ldots;\,z_1=v+v_1t+\ldots; \,z_2=u+t;\,z_4=1.$ From \cite[Lemma 7.74]{HKT}, the conic $\cC^2$ of equation
$(1+uv)^q+v^qX+u^qY+XY=0$ is the hyperosculating conic of $\cC_q$ at $P$ with $I(P,\cC_q\cap \cC^2)\geq q$, and this intersection number is bigger than $q$ if and only if
$(v+v_1u)^q+v_1^{q}u+v=0$. The latter equation is equivalent to $c^{-1}{\rm{Tr}}(u)^2\in \mathbb{F}_q$. More precisely,
if $I(P,\cC_q\cap \cC^2)>q$ occurs then $I(P,\cC_q\cap \cC^2)=q+1$. In fact, $I(P,\cC_q\cap \cC^2)$ for the hyperosculating conic of $\cC_q$ at $P$ can exactly be computed using the above primitive branch representation. From
\begin{equation} \label{con} (1+uv)^q+v^q(u+t)+u^q(v+v_1t+v_2t^2+ \ldots) +u(v+v_1t+v_2t^2+ \ldots)+t(v+v_1t+v_2t^2+ \ldots)=0. \end{equation}
the coefficients of $t$ and $t^i$ for $i \geq 2$ in \eqref{con} are given respectively by
\begin{equation} \label{coeffcon} {\mbox{$v^q+u^qv_1+uv_1+v$ and $u^qv_i+uv_i+v_{i-1},$  for $i \geq 2$.}} \end{equation}
From \eqref{coefq2}, $u^qv_i+uv_i+v_{i-1}=0$ holds for $i \leq q-1$. Moreover, from \eqref{coefq}, $u^qv_q+uv_q+v_{q-1}=0$ holds if and only if  $$\frac{-v_{q-1}-(v^q+v)-v_1^q(u^q+u)}{u^q+u} = v_q=-\frac{v_{q-1}}{u^q+u},$$ and hence if and only if $(v^q+v)/(u^q+u)=c/(u^q+u)^2 \in \fq$. Thus, $I(P,\cC_q\cap \cC^2) \geq q+1$ if and only if $c^{-1}{\rm{Tr}} (u)^2 \in \fq$. Assume that $c^{-1}{\rm{Tr}}(u)^2 \in \fq$. From \eqref{coefq1}, $u^{q}v_{q+1}+uv_{q+1}+v_{q}=0$ is equivalent to  $$v_1^q+v_1=0,$$ and hence to $(c^{-1}{\rm{Tr}}(u)^2)^{q-1}=-1$, a contradiction to $c^{-1}{\rm{Tr}}(u)^2 \in \fq$. Thus, $I(P,\cC_q\cap \cC^2)=q+1$ if and only if $c^{-1}{\rm{Tr}} (u)^2 \in \fq$.

Therefore, the order sequence of $\cC_q$ at $P$, with respect to conics, is $(0,1,2,q)$ apart from the case $c^{-1}{\rm{Tr}}(u)^2\in \mathbb{F}_q$ where the last order is $q+1$.
\end{proof}
\begin{lemma}
\label{lem68dic2016}  Assume that $q\geq 5$. With respect to conics, the Artin-Schreier-Mumford curve $\cC_q$ is Frobenius non-classical over $\mathbb{F}_{q^2}$ has order sequence $(0,1,q)$.
\end{lemma}
\begin{proof}
In fact, if $P(u,v)\in \cC_q$, its image $P'(u^{q^2},v^{q^2})$ by the Frobenius map over $\mathbb{F}_{q^2}$ is also a point of $\cC_q$. This follows from the fact that the left hand side in
(\ref{eq17dic2016}) can also be written as $z_0+z_1X^q+z_2Y^q+z_3X^q+z_4Y^qY^q+z_5Y^q$ with $z_0,\ldots,z_5$ as in the proof of Lemma \ref{lem58dic2016}. This together with Lemma \ref{lem58dic2016} also show that $q$ is a Frobenius order of $\cC_q$  over $\mathbb{F}_{q^2}$ which shows the second assertion.
\end{proof}

\section{Adjoint curves of $\cC_q$ of degree $q$}
\label{adj}
In our proof the divisors ${\bf{P}}=P_1+\ldots P_q$ and ${\bf{Q}}=Q_1+\ldots Q_q$ of $\mathbb{K}(\cX)$, and their sum ${\bf{G}}={\bf{P}}+{\bf{Q}}$ play an important role. They define the Riemann-Roch space $\mathcal{L}({\bf{G}})$ and we begin by computing its dimension $\ell({\bf{G}})$.

For this purpose we use a geometric approach by means of the corresponding complete linear series $|\verb"G"|$ on $\cC_q$ viewed as a projective curve
with homogeneous equation $(X_1^q+X_1X_3^{q-1})(X_2^q+X_2X_3^{q-1})-cX_3^{2q}=0$; see \cite[Chapter 3]{Goppa} and \cite[Chapter 6.2]{HKT}. Since $\cC_q$ is a plane curve with only ordinary singularities, the divisors of $|\verb"G"|$ are cut out on $\cC_q$ by certain adjoint curves of a given degree $d$ which are determined in our case as follows. Let ${\bf{D}}=(q-1){\bf{G}}$ be double-point divisor of $\cC_q$. Then a plane curve $\cF$ is adjoint if (the intersection divisor) $\cF\cdot\cC_q$  has the property that $\cF\cdot\cC_q \succeq{\bf{D}}$.

Take any adjoint $\cF$ of degree $d$ such that $\cF\cdot\cC_q \succeq {\bf{D}}+{\bf{G}}$ and let ${\bf{B}}=\cF\cdot\cC_q-({\bf{D}}+{\bf{G}})$. The adjoint curves $\Phi$ with $\deg(\Phi) = l$ such that $\Phi \cdot \cC_q \succeq {\bf{D}}+{\bf{B}}$ form a linear system that contains a linear subsystem $\Lambda$ free from curves having $\cC_q$ as a component. The curves in $\Lambda$ are the adjoints of $|\verb"G"|$. The (projective) dimension of $|\verb"G"|$ is $\dim(\Lambda)$, that is, the maximum number of linearly independent curves in $\Lambda$. Here $\ell({\bf{G}})=\dim(\Lambda)+1$.

In our case, a good choice for $\cF$ is the curve of degree $d=q$ which is the line at infinity (of equation $X_3=0$) counted $q$ times. In fact, $\cF\cdot \cC_q=q {\bf{G}}=
{\bf{D}}+{\bf{G}}$, and hence ${\bf{B}}$ is the zero divisor. Therefore $\Lambda$ comprises all degree $q$ adjoints  of $\cC_q$, equivalently all plane curves such that both $X_\infty$ and $Y_\infty$ are (at least) $(q-1)$-fold points. Also, $\deg|\verb"G"|=2q^2-2q(q-1)=2q$
\begin{lemma}
\label{lem17dice2016} Every degree $q$ adjoint of\, $\cC_q$ splits into the line at infinity counted $(q-2)$ times and a conic through $X_\infty$ and $Y_\infty$, and the converse also holds.    \end{lemma}
\begin{proof} Take a projective plane curve $\Phi$ of degree $q$ with affine equation $F(X,Y)=\sum_{i=0}^q F_i(X,Y)$ where $F_i(X,Y)$ is either $F_i(X,Y)=0$, or a homogeneous polynomial of degree $i$. If $\Phi$ is adjoint then $Y_\infty$ is a $(q-1)$-fold point of $\Phi$ and every (vertical) line of equation $X=d$ meets $\Phi$ in zero or just one affine point. In the latter case the line is not tangent to $\Phi$ at the common affine point. Therefore the polynomial $f_d(Y)=F_i(d,Y)$ has at most one (simple) root. The same holds for the polynomials $g_d(X)=F_i(X,d)$. Therefore, $F_i(X,Y)$ vanishes for $i\geq 3$ while $F_2(X,Y)=dXY$ whence $F(X,Y)=a+b_1X+b_2Y+dXY$. This shows that the line at infinity is a $(q-2)$-fold component of $\Phi$, and that the other component is a (possibly degenerate) conic through the points $X_\infty$ and $Y_\infty$.
\end{proof}
From Lemma \ref{lem17dice2016}, the complete linear series $|\verb"G"|$ gives rise to a projective curve $\cY$ of degree $2q$ in the three dimensional projective space $PG(3,\mathbb{K})$ with homogeneous coordinates $(Y_1,Y_2,Y_3,Y_4)$. More precisely, $\cY$ is birationally isomorphic to $\cX$ by the quadratic morphism $\tau:(X_1,X_2,X_3)\mapsto (X_1X_3,X_2X_3,X_1X_2,X_3^2)$. We go on to determine the images of the branches of $\cC_q$ centered at $X_\infty$ or at $Y_\infty$. A straightforward computation shows that if the branch  of $\cC_q$ is centered at $X_\infty$ and tangent to the line $Y=d$ with ${\rm{Tr}}(d)=0$, then a primitive branch representation is $X_1=1,\,X_2=dt+ct^{q+1}-ct^{2q}+\ldots,X_3=t$. Hence its image by $\tau$ has a primitive branch representation $Y_1=1,\,Y_2=dt+ct^{q+1}-ct^{2q}+\ldots,Y_3=d+ct^q-ct^{2q-1}+\ldots,Y_4=t$. Therefore, the image of the branch is centered at the point $P_d=(1,0,d,0)$ with order sequence $(0,1,q,q+1)$.
The $q$ points $P_d$ are pairwise distinct and form a set $\Omega_1'$ of collinear points. Similarly, the branch of $\cC_q$ centered at $Y_\infty$ and tangent to the line $X=d$ with ${\rm{Tr}}(d)=0$ is mapped by $\tau$ to a branch of $\cY$ centered at the point $Q_d=(0,1,d,0)$ with order sequence $(0,1,q,q+1)$, and the set $\Omega_2'$ of these points $Q_d$ has size $q$ and consists of collinear points. The intersection of $\cY$ with the plane $H$ at infinity of $PG(3,q)$ contains $\Omega_1'\cup \Omega_2'$. Since the latter set has size $2q$ and equals the degree of $\cY$, the intersection multiplicity $I(P_d,\cY\cap H)=1$. This shows that each point in $\Omega_1'\cup \Omega_2'$ is a simple point of $\cY$.

Each other point of $\cY$ arises from a (simple) affine point of $\cC_q$. From the proof of Lemma \ref{lem58dic2016} if the branch of $\cC_q$ is centered at the affine point $(u,v)$ with $(u^q+u)(v^q+v)=c$ then a primitive branch representation is $${\mbox{$X_1=u+t,\,X_2=v+v_1t+v_2t^2+v_3t^3+\ldots, X_3=1$ with $v_i=(-1)^i(v^q+v)/(u^q+u)^i$ for $i=1,2,3, \ldots, q-1$}}.$$
Therefore, the image by $\tau$ is a branch of $\cY$ centered at the point $R=(u,v,uv,1)$ with a primitive representation
$$Y_1=u+t,\,Y_2=v+v_1t+v_2t^2+v_3t^3+\ldots,\,Y_3=uv+(uv_1+v)t+(uv_2+v_1)t^2+\ldots,\, Y_4=1.$$
Thus, $I(R,\cH\cap\cY)=1$ for the hyperplane of equation $Y_1=uY_0$ while $I(R,\cH\cap\cY)=2$ for the hyperplane of equation $Y_2-vY_0=v_1(Y_1-uY_0)$. This shows that $0,1,2$ are orders of $\cY$ at $R$. Furthermore, $I(R,\cH\cap\cY)\geq q$ for the hyperplane arising from the hyperosculating conic of $\cC_q$ at $P(u,v)$, that is, for the hyperplane of equation $(1+uv)^qY_4^2+v^qY_1Y_4+u^qY_2Y_4+Y_1Y_2=0$. As we have pointed out in the proof of Lemma \ref{lem58dic2016}, $I(R,\cH\cap\cY)=q$ apart from a finite number of points $R$ arising from affine points of $\cC_q$ where the order sequence is $(0,1,2,q+1)$. Therefore, the order sequence of $\cY$ is $(0,1,2,q)$, and the third order at a point of $\cY$ is bigger than $2$ if and only if that point arises from a branch of $\cC_q$ centered at $X_\infty$ or $Y_\infty$.

Since different points $(u,v)$ are taken by $\tau$ to different points, the curve $\cY$ is nonsingular, and branches may be identified by their centers.
\begin{theorem}
\label{th1} The Artin-Schreier-Mumford curve has a nonsingular model $\cY$ in $PG(3,\mathbb{K})$. The curve $\cY$ has degree $2q$ and if $q\geq 5$ then it is non-classical with order sequence $(0,1,2,q)$ and also Frobenius non-classical over $\mathbb{F}_{q^2}$ with order sequence $(0,1,q)$.
\end{theorem}

\section{The automorphism group of the Artin-Schreier-Mumford curve}
From the discussion in Section \ref{adj}, the linear series cut out on $\cY$ by planes is complete. By \cite[Theorem 11.18]{HKT}, this yields that $\aut(\cX)$ can be represented by a linear group $G$ of $PG(3,\mathbb{K})$ which leaves the set of point of $\cY$ invariant. Furthermore, every map in $G$  preserves the order sequence at any point of $\cY$.
This proves the following result.
\begin{lemma}
\label{lem37dice2016} $G$ leaves $\Omega_1'\cup \Omega_2'$ invariant.
\end{lemma}
More precisely, the following result holds.
\begin{lemma}
\label{lem77dice2016} $G$ contains a subgroup of index $\le 2$ that leaves both $\Omega_1'$ and $\Omega_2'$ invariant.
\end{lemma}
\begin{proof} As we have pointed out in Section \ref{adj}, both $\Omega_1'$ and $\Omega_2'$ consist of collinear points. Since their size is equal to $q>2$, any linear map leaving $\Omega_1'\cup \Omega_2'$ invariant acts on it  either interchanging $\Omega_1'$ with $\Omega_2'$, or leaving both invariant. Those maps in $G$ which act as in the latter case form a subgroup of index $\le 2$.
\end{proof}
We are in a position to prove the following result.
\begin{theorem}
\label{th2} The automorphism group of the Artin-Schreier-Mumford curve has order $2q^2(q-1)$ and it is the semidirect product $\Delta\rtimes D_{q-1}$ where $\Delta$ is an elementary abelian (normal) subgroup of order $q^2$, and $D_{q-1}$ is a dihedral group of order $2(q-1)$.
\end{theorem}
\begin{proof} By Lemma \ref{lem77dice2016}, a subgroup $H$ of $G$ of index $\leq 2$ leaves both $\Omega_1'$ and $\Omega_2'$ invariant. As we have shown in Section \ref{adj}, the points of $\Omega_1'$ lie on the line $\ell_1$ of equation $X_2=X_4=0$ while those of $\Omega_2'$ on the line $\ell_2$ of equation $X_1=X_4=0$. These lines $\ell_1$ and $\ell_2$ are incident as $Z_\infty(0,0,1,0)$ is their common point. Hence  $H$ fixes $Z_\infty$.

Take $h\in H$ with its $4\times 4$ matrix representation $M=(m_{ij})$. Then $m_{13}=m_{23}=m_{03}=0$ as $h$ fixes $Z_\infty$. Also, $m_{12}=m_{42}=0$ as $h$ leaves $\Omega_1'$ invariant. Similarly, $m_{21}=m_{41}=0$ as $h$ leaves $\Omega_2'$ invariant. Therefore,
$$M=\left(
    \begin{array}{cccc}
      m_{11} & 0 & 0 & m_{14} \\
      0 & m_{12} & 0 & m_{24} \\
      m_{31} & m_{32} & m_{33} & m_{34}\\
      0 & 0 & 0  & 1 \\
    \end{array}
  \right),
$$
which shows that $h$ induces a linear map on $\cC_q$. Therefore $h\in \Delta\rtimes C_{q-1}$, and hence $H=\Delta\rtimes C_{q-1}$. Since $[G:H]\leq 2$ and $[\aut(\cX):(\Delta\rtimes C_{q-1})]=2$,
this yields $[G:H]=2$ and $\aut(\cX)=\Delta\rtimes D_{q-1}$.
\end{proof}

\vspace{0.2cm}\noindent G\'abor KORCHM\'AROS and Maria MONTANUCCI\\ Dipartimento di
Matematica, Informatica ed Economia\\ Universit\`a degli Studi  della Basilicata\\ Contrada Macchia
Romana\\ 85100 Potenza (Italy).\\E--mail: {\tt
gabor.korchmaros@unibas.it and maria.montanucci@unibas.it}

    \end{document}